\documentclass[12pt]{article}

\usepackage[T1]{fontenc}
\usepackage{avant}
\usepackage[cp1250]{inputenc}
\usepackage{amsmath,amssymb,amsthm}
\usepackage{amscd}
\newtheorem{theorem}{Theorem}
\newtheorem{proposition}{Proposition}
\newtheorem{lemma}{Lemma}
\newtheorem{definition}{Definition}

\newtheorem{remark}{Remark}

\title{On a classification of fat bundles over compact homogeneous spaces}

\author{Maciej Boche\'nski, Anna Szczepkowska, Aleksy Tralle,\\
and Artur Woike}

\begin{document}
\maketitle{}

\abstract{This article deals with fat bundles. B\'erard-Bergery classified all homogeneous  bundles of that type. We ask a question of a possibility to generalize his description in the case of arbitrary $G$-structures over homogeneous spaces. We obtain necessary conditions for the existence of such bundles. These conditions yield a kind of classification of fat bundles associated with $G$-structures over compact homogeneous spaces provided that the connection in a $G$-structure is canonical.}

\section{Introduction}
Let $P(M,G)$ be a principal bundle endowed with a connection form $\theta$ and its curvature form $\Omega$. Let $\operatorname{Ker}\theta=\mathcal{H}\subset TP$ be the corresponding horizontal distribution. Assume that the Lie algebra $\mathfrak{g}$ of $G$  is endowed with an invariant non-degenerate bilinear form $B_{\mathfrak{g}}$. We say that a vector $u\in\mathfrak{g}$ is {\it fat} or that the connection form $\theta$ is $u$-fat, if the bilinear two-form 
$B_{\mathfrak{g}}(\Omega(\cdot ,\cdot ),u)$  
is non-degenerate on $\mathcal{H}$. If the fatness condition is fulfilled for every non-zero $u\in \mathfrak{g}$ then we say that the connection form $\theta $ is fat or that the principal bundle is fat.
The role of the fatness condition in Riemannian geometry follows from its relation with the O'Neil tensor \cite{B}, \cite{Z} of a specific fiberwise metric on the associated bundle. In greater detail, consider an associated bundle
 $$F\rightarrow P\times_GF\rightarrow M.$$
 Endow $F$ with a $G$-invariant Riemannian metric $g_F$, and $M$ with a Riemannian metric $g_M$. Equip $P\times _G F$ with the {\it connection metric} defining it to be the Riemannian metric which equals $g_F$ on $F$,  $g(X^*,Y^*)=g_M(X,Y)$ for horizontal lifts $X^*,Y^*$ of $X,Y\in TM$ respectively, and  declaring $TF$ and $\mathcal{H}$ to be orthogonal with respect to $g$.
For this  metric the following holds. Let $A_X$ denote the O'Neil tensor. 
 \begin{theorem}{\rm \cite{Z}}\label{thm:o'neil} The connection metric $g$ on $P\times_GF$ is complete and defines a Riemannian submersion $\pi: P\times_GF\rightarrow M$ with totally geodesic fiber $F$ and with holonomy group a subgroup of $G$. Conversely, every Riemannian submersion over $M$ with totally geodesic fibers arises in this fashion. Moreover, the O'Neil tensor for such submersion satisfies the equality
 $$\theta(A_XY)=-\Omega(X,Y).$$
 \end{theorem}
 The condition of fatness is an important tool of constructing manifolds of positive and non-negative curvature \cite{GZ},\cite{W}. 
 In the Riemannian context, the following definition of fatness is used: a Riemannian submersion $\pi: E\rightarrow M$ with totally geodesic fibers is fat, if $A_XU\not=0$ for all horizontal vector fields $X$ and vertical vector fields $U$ (``all vertizontal curvatures are positive''). 
 For the associated bundles, the characterization of fatness of any connection is known \cite{Z} and can be formulated as follows.
 \begin{theorem}{\rm \cite{Z}}\label{thm:asso} The Riemannian submersion $P\times_GF\rightarrow M$ with totally geodesic fibers $G/L$ is fat if and only if the 2-form
 $B_{\mathfrak{g}}(\Omega(X,Y),u)$
 is non-degenerate on the horizontal distribution for all $u\in \mathfrak{l}^{\perp}$.
 \end{theorem}
 Keeping the above theorem in mind, from now on whenever we say that an associated bundle $G/L \rightarrow P/L \rightarrow M$ is fat, we mean that the set $\frak{l}^\perp $ consists of fat vectors.
 In this paper we consider associated bundles  with homogeneous fibers since it was proved (see Proposition 2.6 in \cite{Z}) that every fat submersion necessarily has a homogeneous fiber. 
 Note that the connection metric $g$ depends on  a (chosen) principal connection. The  following is known. 
 \begin{enumerate}
 \item There is a an algebraic condition on the curvature tensor of the associated metric connection of the sphere bundles of the form
 $$SO(n+1)/SO(n)\rightarrow P/SO(n)\rightarrow M$$
 ensuring fatness (see Proposition 2.21 in \cite{Z}).
 \item A theorem in \cite{DR} shows that the only fat $SO(4)/SO(3)$-bundle over $S^4$ is the Hopf bundle $S^7\rightarrow S^4$.
 \item A theorem of B\'erard-Bergery \cite{BB} which classifies all homogeneous fat bundles, that is, associated bundles of the form $H/L\rightarrow K/L\rightarrow K/H$, where $K,H,L$ are compact Lie groups. Note that in this case the classification is obtained for {\it any} invariant connection.
 \item There are  necessary conditions for fatness, see for example \cite{FZ}.
 \end{enumerate}

 To stress the fact that in general the fatness condition is dependent on the choice of the connection, we will always refer to ``fatness with respect to a connection''. Here it seems to be instructive to compare \cite{BB} with the general problem. In the case of principal bundles $K\rightarrow K/H$, there is a one-to-one correspondence between the invariant connections and the linear maps $\Lambda:\mathfrak{k}\rightarrow\mathfrak{h}$ satisfying the conditions 
 $\Lambda (X)=X$ for any $X\in \mathfrak{h}$, and $\Lambda ([X,Y])=[X,\Lambda (Y)]$  for all $X\in\mathfrak{h}, Y\in \mathfrak{k}$. If, for example, $\mathfrak{k}$ is semisimple the fatness condition can be expressed as follows: for any $X\in\mathfrak{h}^{\perp}$ and any $Y\in \operatorname{Ker}\Lambda$ there exists $Z\in\operatorname{Ker}\Lambda$ such that $\langle X,\Lambda ([Y,Z])\rangle\not=0$. Here $\langle \cdot , \cdot \rangle$ denotes the Killing form. The latter condition can be expressed entirely in terms of the Lie brackets of $\mathfrak{k}$. In general, the problem becomes much more complicated even for the invariant connections in $G$-structures. In this work we will adopt the following terminology. A homogeneous bundle
 $$H/L\rightarrow K/L\rightarrow K/H$$
 endowed with an invariant fat connection will be called the {\it B\'erard-Bergery} bundle. The triples $(K,H,L)$ (classified in \cite{BB}) will be called the {\it B\'erard-Bergery triples}, and the pairs $(H,H\cap L)$ will be called the {\it B\'erard-Begery pairs}. Note that all $(K,H,L)$ and $(H,H\cap L)$ are known \cite{BB} (see Theorem 2 in this paper). In this work we use this classification, therefore, we reproduce part of it in the last section.
 In this article we want to follow the line of reasoning of \cite{BB} and ask a more general question: {\it can one classify  bundles associated with a $G$-structure over a compact homogeneous space $K/H$ which are fat with respect to invariant connections?}
The purpose of this paper is to solve this problem for the class of fat bundles over homogeneous spaces determined by the canonical invariant connections (Theorem \ref{thm:main}). Our result  yields  necessary conditions, which can be considered as a kind of classification, because we show that the bases of such bundles are  determined by the B\'erard-Bergery triples, and they are determined by a fiber bundle, which is naturally related to the B\'erard-Bergery bundle.  For example, if $G$ is simple, there are no other fat bundles than the B\'erard-Bergery homogeneous bundles. Consider the B\'erard-Bergery pair $(H,H\cap L)$. Let $\mathfrak{n}'$ be a maximal common ideal in $\mathfrak{h}$ and $\mathfrak{h}\cap\mathfrak{l}$. Then, clearly, 
$$\mathfrak{h}=\mathfrak{h}'\oplus\mathfrak{n}',\mathfrak{h}\cap\mathfrak{l}=(\mathfrak{h}\cap\mathfrak{l})'\oplus\mathfrak{n}'.$$
Analogously, let $\mathfrak{n}$ be a common maximal ideal in $\mathfrak{g}$ and $\mathfrak{l}$.
\begin{theorem}\label{thm:main} Let $G\rightarrow P\rightarrow K/H$ be a $G$-structure over a homogeneous space $K/H$ of a compact semisimple Lie group $K$. Assume that $G$ is compact and semisimple. Let 
$$G/L\rightarrow P/L\rightarrow K/H$$
be an associated bundle. If it is fat with respect to the canonical connection, then:
\begin{enumerate}
\item $\mathfrak{g}=\mathfrak{h}+\mathfrak{l}$ and $(\mathfrak{h},\mathfrak{h}\cap\mathfrak{l})$ is a B\'erard-Bergery pair;
\item $K/H$ is a symmetric space of maximal rank;
\item $\mathfrak{h}=\mathfrak{h}'\oplus\mathfrak{n}'$ and $\mathfrak{g}=\mathfrak{h}'\oplus\mathfrak{n}$
\end{enumerate}
In particular, if $G$ is simple, then the associated bundle is fat with respect to the canonical connection only if $G=H, P=K$. Therefore, the fiber bundle is a B\'erard-Bergery bundle.
\end{theorem}

Let us mention a recent publication \cite{FP1}, where the authors consider a particular class of fat bundles over 4-manifolds $M$ endowed with an $S^1$-action which lifts to an $SO(3)$-bundle and leaves the connection invariant. Thus, the case considered in \cite{FP1} is an opposite to the case we deal with. It would be interesting to get deeper understanding of the symmetry of fat connections. Finally, let us mention that there is a different but related notion of symplectic fatness \cite{W}, and there is an interplay between these two \cite{FP},\cite{bstw}, \cite{KTW}.

\section{Preliminaries}
\subsection{Notation, terminology and assumptions}
The basic tools in this work are invariant connections and the Lie group theory. Therefore, we closely follow the terminology and notation in \cite{KN1},\cite{KN2}, \cite{H} and \cite{OV} without further explanations. We denote the Lie algebras of Lie groups by the corresponding Gothic letters, so the Lie algebra of the Lie group $G$ is denoted by $\mathfrak{g}$, etc. We consider principal bundles $G\rightarrow P\rightarrow M$ and in the most cases use the notation $P(M,G)$. Throughout the whole article we consider homogeneous spaces $K/H$ and assume that both $K$ and $H$ are compact Lie groups. 
In this note we will use the following result.
\begin{proposition} \label{prop:bb}{\rm \cite{BB},\cite{Z}}. If $H/L\rightarrow K/L\rightarrow K/H$ is a fat bundle with $\dim\,(H/L)>1$, then $K/H$ is a Riemannian symmetric space and $\operatorname{rank}\,K=\operatorname{rank}\,H$.
\end{proposition}

\subsection{ Invariant connections on reductive homogeneous spaces}
In this subsection we will recall the theory of invariant connections on homogeneous spaces in the form presented in Sections 1 and 2 of Chapter X of \cite{KN2}. Let $M$ be a smooth manifold of dimension $n$, and let
$$G\rightarrow P\rightarrow M $$
be a $G$-structure, that is, a reduction of the frame bundle $L(M)\rightarrow M$ to a Lie group $G$. Any diffeomorphism $f\in \operatorname{Diff}(M)$ acts on $L(M)$ by the formula 
$$f(u):=(df_xX_1,...,df_xX_n)$$
for any frame $u=(X_1,...,X_n), X_i\in T_xM$ over a point $x\in M$. By definition, $f$ is called an automorphism of the given $G$-structure, if this action commutes with the action of $G$.

Let $M=K/H$ be a homogeneous space of a connected Lie group $K$.  Assume that $M$ is equipped with a $K$-invariant $G$-structure. The latter means that any left translation $\tau(k): K/H\rightarrow K/H$, $\tau(k)(aH)=kaH$ lifts to an automorphism. Let $o=H\in K/H$. Consider the linear isotropy representation of $H,$ that is, a homomorphism $H\rightarrow GL(T_oM)$ given by the formula
$$h \mapsto d\tau(h)_o,\,\text{for}\,h\in H,o=H\in K/H.$$ 
It is important to observe that we can fix a frame $u_o:\mathbb{R}^n\rightarrow T_oM,$ $u_{o}\in P$ and identify the linear isotropy representation of $H$ with a homomorphism $\lambda :H \rightarrow G$
$$\lambda(h)=u_o^{-1}d\tau(h)_ou_o,\,h\in H.$$
One can see this as follows. Denote by $P_{o}\subset P$ the fiber over the point $o,$ then $u_{o} \in P_{o}$.  The action of $H$ lifted to $P$ preserves $P_{o}$, hence  $h(u_{o}) \in P_{o}.$ Since the structure group $G$ acts transitively on $P_{o}$, there exists exactly one $g\in G$ such that
$$h(u_{o})=u_{o}g.$$
It is easy to see that $\lambda (h)=g.$ 

In the sequel we assume that $K/H$ is reductive. In this case $\mathfrak{k}$ can be decomposed into a direct sum
$$\mathfrak{k}=\mathfrak{h}\oplus\mathfrak{m}$$
such that $Ad_{H}(\mathfrak{m})\subset\mathfrak{m}$. Note that the latter implies $[\mathfrak{h},\mathfrak{m}]\subset\mathfrak{m}$. Note that since we consider compact Lie groups, the reductivity can be assumed (see, for example, \cite{KN2}, Chapter X). 
Also, we assume that the isotropy representation is faithful. Let us make one more straightforward but important observation.  One can  identify $\lambda$ with the restriction of the adjoint representation of $H$ on $\mathfrak{m}$  (which we also denote by $\lambda$). 

We say that a connection $\theta$ in $P\rightarrow M$ is $K$-invariant, if for any $k\in K$ the lift of $\tau(k)$ preserves it. We need the following description of the set of invariant connections in the principal bundle $P\rightarrow M$ from \cite{KN2}, Chapter X.

\begin{theorem}\label{thm:inv-connections} Let $M=K/H$ be a reductive homogeneous space equipped with a $K$-invariant $G$-structure. Then there is a one-to-one correspondence between the $K$-invariant connections in it, and $Ad_H$-invariant linear maps
$$\Lambda: \mathfrak{k}\rightarrow \mathfrak{g}, \Lambda|_{\mathfrak{h}}=\lambda.$$
\end{theorem}

\noindent Let $\Lambda_{\mathfrak{m}}=\Lambda |_{\mathfrak{m}}$. Note that here $Ad_H$-invariance means that
$$\Lambda_\mathfrak{m}(Ad\,h(Z))=\lambda(h)(\Lambda_\mathfrak{m}(Z)),\,Z\in\mathfrak{m},h\in H.$$
\begin{definition} {\rm Recall that a connection in the given $K$-invariant $G$-structure is called {\it canonical} if it corresponds to the map $\Lambda$ with $\Lambda_{\mathfrak{m}}=0$. }
\label{r2}
\end{definition}

In this article we will need the description of the holonomy Lie algebra of the canonical connection.
\begin{theorem}{\rm \cite{KN2}}\label{thm:holonomy-alg} The holonomy Lie algebra of the canonical connection of a reductive homogeneous space $K/H$ is given by the formula
$$\mathfrak{h}^*=\lambda(ad([X,Y]_{\mathfrak{h}})),\,\forall\, X,Y\in\mathfrak{m}.$$
\end{theorem}
\subsection{Factorizations of Lie groups and Lie algebras}
Following Onishchik \cite{O},\cite{O1} we say that a triple of Lie algebras $(\mathfrak{g},\mathfrak{h},\mathfrak{l})$ is a {\it factorization}, if $\mathfrak{h}$ and $\mathfrak{l}$ are subalgebras of $\mathfrak{g}$, $\mathfrak{g}=\mathfrak{h}+\mathfrak{l}$, and $\mathfrak{h}\not=\mathfrak{g},\mathfrak{l}\not=\mathfrak{g}$. In the same way, we say that $(G,H,L)$ is a factorization, if $H$ and $L$ are Lie subgroups of $G$, $G=H\cdot L$ and $G\not=H$, $G\not=L$. It is proved in \cite{O},\cite{O1} that, for compact Lie groups, $(G,H,L)$ is a factorization if and only if $(\mathfrak{g},\mathfrak{h},\mathfrak{l})$ is a factorization. All factorizations of simple compact Lie algebras are classified in \cite{O1}. The proof of our result is essentially based on this classification, and we reproduce it here. 
 As usual, we consider the complexification $\mathfrak{g}^c$ and a Cartan subalgebra $\mathfrak{j}\subset\mathfrak{g}^c$. Consider the corresponding root system $\Delta$, and choose the set of simple roots $\alpha_1,...,\alpha_m$. It is known that any irreducible linear representation of $\mathfrak{g}$ is determined (up to an equivalence) by the highest weight of the representation. This is a vector $\Phi\in\mathfrak{j}$, which can be described by the integers 
 $$\Phi_i={2\langle\Phi,\alpha_i\rangle\over\langle\alpha_i,\alpha_i\rangle},\,i=1,...,m.$$
 We write $\Phi=(\Phi_1,...,\Phi_m)$ and denote by $\varphi_i$ the irreducible representation of a simple Lie algebra $\mathfrak{g}$ with the highest weight $\Phi$ which has $\Phi_i=1$ and $\Phi_j=0,i\not=j$. The following classification result is proved by Onishchik.
 \begin{theorem}\label{thm:onishchik}{\rm \cite{O}} Let $\mathfrak{g}$ be a compact simple Lie algebra. All factorizations  
 $(\mathfrak{g},\mathfrak{l},\mathfrak{h})$ and possible embeddings 
 $$i': \mathfrak{h}\rightarrow\mathfrak{g}, i'':\mathfrak{l}\rightarrow\mathfrak{g}$$ are given in Table 1.
 \end{theorem} 
\begin{center}
\begin{tabular}{|c|c|c|c|c|c|c|}
\hline
$\mathfrak{g}$ & $\mathfrak{h}$ & $i'$ & $\mathfrak{l}$ & $i''$ & $\mathfrak{h} \cap \mathfrak{l}$  & restrictions \\
\hline \hline
$A_{2n-1}$ & $C_n$ & $\varphi _1$ & $A_{2n-2}$ & $\varphi _1+N$ & $C_{n-1}$ & $n>1$\\  
\hline
$A_{2n-1}$ & $C_n$ & $\varphi _1$ & $A_{2n-2}\oplus T$ & $\varphi _1+N$ & $C_{n-1}\oplus T$ & $n>1$\\
\hline
$B_3$  & $G_2$ & $\varphi _2$ & $B_2$ & $\varphi _1+2N$ & $A_{1}$ & \null \\
\hline
$B_3$  & $G_2$ & $\varphi _2$ & $B_2\oplus T$ & $\varphi _1+2N$ & $A_{1}\oplus T$ & \null \\
\hline
$B_3$  & $G_2$ & $\varphi _2$ & $D_3$ & $\varphi _1+N$ & $A_{2}$ & \null \\
\hline
$D_{n+1}$  & $B_n$ & $\varphi _1+N$ & $A_n$ & $\varphi _1+\varphi _n$ & $A_{n-1}$ & $n>2$\\
\hline
$D_{n+1}$  & $B_n$ & $\varphi _1+N$ & $A_n\oplus T$ & $\varphi _1+\varphi _n$ & $A_{n-1}\oplus T$ & $n>2$ \\
\hline
$D_{2n}$  & $B_{2n-1}$ & $\varphi _1+N$ & $C_n$ & $\varphi _1+\varphi _1$ & $C_{n-1}$ & $n>1$\\
\hline
$D_{2n}$  & $B_{2n-1}$ & $\varphi _1+N$ & $C_n\oplus T$ & $\varphi _1+\varphi _1$ & $C_{n-1}\oplus T$ & $n>1$ \\
\hline
$D_{2n}$  & $B_{2n-1}$ & $\varphi _1+N$ & $C_n\oplus A_1$ & $\varphi _1+\varphi _1$ & $C_{n-1}\oplus A_{1}$ & $n>1$ \\
\hline
$D_8$  & $B_7$ & $\varphi _1+N$ & $B_4$ & $\varphi _4$ & $B_{3}$ & \null \\
\hline
$D_{4}$  & $B_3$ & $\varphi _3$ & $B_2$ & $\varphi _1+3N$ & $A_{1}$ & \null \\
\hline
$D_{4}$  & $B_3$ & $\varphi _3$ & $B_2\oplus T$ & $\varphi _1+3N$ & $A_{1}\oplus T$ & \null\\
\hline
$D_{4}$  & $B_3$ & $\varphi _3$ & $B_2\oplus A_1$ & $\varphi _1+3N$ & $A_{1}\oplus A_{1}$ & \null \\
\hline
$D_{4}$  & $B_3$ & $\varphi _3$ & $D_3$ & $\varphi _1+2N$ & $A_{2}$ & \null \\
\hline
$D_{4}$  & $B_3$ & $\varphi _3$ & $D_3\oplus T$ & $\varphi _1+2N$ & $A_{2}\oplus T$ & \null \\
\hline
$D_{4}$  & $B_3$ & $\varphi _3$ & $B_3$ & $\varphi _1+N$ & $G_{2}$ & \null \\
\hline
\end{tabular}
\end{center}
\vskip6pt
\centerline{TABLE 1}    
\begin{remark} {\rm In Table 1 above $N$ stands for the trivial representation. The types of simple Lie algebras are denoted as usual, following \cite{OV}}.  
\end{remark}
\section{Proof of Theorem \ref{thm:main}}
\subsection{Preliminary statements}
\begin{proposition}\label{prop:onishchik}
Let $P(M,G)$ be a principal fiber bundle with a connection form $\theta$ and the curvature form $\Omega$. Assume that $M=K/H$ is a homogeneous space of a compact Lie group $K$.  Fix a point $p\in P$. Let $\mathfrak{h}'$ be the Lie subalgebra of $\mathfrak{g}$ generated by all $\Omega_p(X,Y), X,Y\in\mathcal{H}_p$. Then, if there exists a fat associated bundle with fiber $G/L$, then $\mathfrak{g}=\mathfrak{h}'+\mathfrak{l}$. Moreover, if the bundle is fat with respect to the canonical connection, then $\mathfrak{g}=\mathfrak{h}+\mathfrak{l}$.
\end{proposition}
\begin{proof}
Notice that the fatness condition  is equivalent to the condition that all 2-forms
$B_{\mathfrak{g}}(\Omega_p(X,Y),u)$
are non-degenerate for all $u\in\mathfrak{l}^{\perp}$,  all $p\in P$ and $X,Y\in \mathcal{H}_p$.  The latter implies
$(\mathfrak{h}')^{\perp} \cap \mathfrak{l}^{\perp}=\{ 0 \} .$
 But
$(\mathfrak{h}')^{\perp} \cap \mathfrak{l}^{\perp}=\{ 0 \} \ \textrm{ iff} \ \mathfrak{g}=\mathfrak{h}'+ \mathfrak{l}. $ We complete the proof by noticing that $\mathfrak{h}'\subset\mathfrak{h}^*$, where $\mathfrak{h}^*$ denotes the holonomy Lie algebra, and that for the canonical connection $\mathfrak{h}^*\subset \mathfrak{h}\cong\lambda (\mathfrak{h})$, by Theorem \ref{thm:holonomy-alg}.
\end{proof}
Let there be given a principal $G$-bundle
$G\rightarrow P \rightarrow K/H$
endowed with a principal $K$-invariant connection $\theta .$ Use Theorem \ref{thm:inv-connections}. On the Lie algebra level we get that a connection $\theta$ is given by a linear mapping $\Lambda:\mathfrak{k} \rightarrow \mathfrak{g}$ such that

\begin{itemize}
	\item $\Lambda (X)=X$ for every $X\in \mathfrak{h},$
	\item $\Lambda [X,Y]=[X,\Lambda (Y)]$ for every $X\in \mathfrak{h}$ and $Y\in \mathfrak{k}$.
\end{itemize}
Here we identify $X$ with  $\lambda(X)$.
Let $L\subset G$ be a closed subgroup and $A:=\mathfrak{l}^{\perp}\subset \mathfrak{g}$ be the orthogonal complement of $\mathfrak{l}$ in $\mathfrak{g}$ with respect to $B_{\mathfrak{g}}$ (where $B_{\mathfrak{g}}$ denotes an invariant, negative-definite 2-form on $\mathfrak{g}$). Let also $B:=\mathfrak{h}\cap \mathfrak{l}$ and $C:=(\mathfrak{h}\cap \mathfrak{l})^{\perp}_{\mathfrak{h}} \subset \mathfrak{h}$ be the orthogonal complement of $\mathfrak{h}\cap \mathfrak{l}$ in $\mathfrak{h}$ with respect to $B_{\mathfrak{g}}.$ Throughout this note we use the following notation: if $W=U\oplus V$, is a direct sum of vector spaces, the projection onto $V$ is denoted by $proj_V$. 

\begin{lemma} If $\mathfrak{g}=\mathfrak{h}+\mathfrak{l}$ then for any subspace $V\subset \mathfrak{k}$ the conditions $(1)$ and $(2)$ below are equivalent
\begin{equation}
\forall_{0\neq X\in A} \ \forall_{0\neq Y\in V} \ \exists_{0\neq Z\in V} \ B_{\mathfrak{g}} (X, [Y,Z]_{\mathfrak{h}})\neq 0.
\label{lr2}
\end{equation}
\begin{equation}
\forall_{0\neq X\in C} \ \forall_{0\neq Y\in V} \ \exists_{0\neq Z\in V} \ B_{\mathfrak{g}} (X, [Y,Z]_{\mathfrak{h}})\neq 0.
\label{lr3}
\end{equation}
\label{lemma1}
\end{lemma}
\begin{proof}
 We have the (orthogonal) direct sum decompositions $\mathfrak{h}=B\oplus C,$ $\mathfrak{g}=\mathfrak{h}\oplus \mathfrak{h}^{\perp}$ and thus $\mathfrak{g}=B\oplus C\oplus \mathfrak{h}^{\perp}$. For any vector $X\in \mathfrak{g}$, denote its $A$-component ($B,C$-component, etc) by $X_{A}$ ($X_{B},X_{C},$ etc).
Note that for $X\in A\subset \mathfrak{g}$ and $H\in \mathfrak{h}$ we have $H=H_{B}+H_{C}$ and
$$B_{\mathfrak{g}} (X,H)=B_{\mathfrak{g}} (X,H_{B}+H_{C})=B_{\mathfrak{g}} (X,H_{C})=B_{\mathfrak{g}} (X_{\mathfrak{h}^{\perp}}+X_{\mathfrak{h}},H_{C})=$$
$$B_{\mathfrak{g}} (X_{\mathfrak{h}},H_{C})=B_{\mathfrak{g}} (X^{B}_{\mathfrak{h}}+X^{C}_{\mathfrak{h}},H_{C})=B_{\mathfrak{g}} (X^{C}_{\mathfrak{h}},H_{C})=B_{\mathfrak{g}} (X_{C},H_{C})=$$
$$B_{\mathfrak{g}} (X_{C},H_{C}+H_{B})=B_{\mathfrak{g}} (X_{C},H)$$
where $X^{B}_{\mathfrak{h}}$ ($X^{C}_{\mathfrak{h}}$) is the $B$-component ($C$-component) of $X_{\mathfrak{h}}.$

Thus the condition (\ref{lr2}) implies that
\begin{equation}
\operatorname{Ker}(proj_C)|_{A}=\{ 0 \}
\label{lr4}
\end{equation}
But
$$\mathfrak{g}=\mathfrak{l}\oplus A=\mathfrak{l}\oplus C$$
are the direct sum decompositions and therefore $\dim A=\dim C$, so $proj_C|_{A}:A\rightarrow C$ is an isomorphism. Thus,
$$B_{\mathfrak{g}} (X, [Y,Z]_{\mathfrak{h}})=B_{\mathfrak{g}} (X_{C}, [Y,Z]_{\mathfrak{h}})=B_{\mathfrak{g}} (proj_{C}(X), [Y,Z]_{\mathfrak{h}})$$
therefore  condition (\ref{lr2}) implies  condition (\ref{lr3})..

Now we will show that condition (\ref{lr3}) implies  condition (\ref{lr2}). Assume that (\ref{lr2}) is not satisfied. Then there exists $0\neq X\in A$ and $0\neq Y\in V$ such that
$$\forall_{0\neq Z\in V} \ B_{\mathfrak{g}} (X, [Y,Z]_{\mathfrak{h}})=0$$
therefore
$$\forall_{0\neq Z\in V} \ B_{\mathfrak{g}} (X_{C}, [Y,Z]_{\mathfrak{h}})=0.$$
If $X_{C}\neq 0$ then the lemma is proved. Assume that $X_{C}=0$. As $X\in A$ thus
$$X=X_{C}+X_{\mathfrak{h}^{\perp}}=X_{\mathfrak{h}^{\perp}},$$
that is $0\neq X\in X_{\mathfrak{h}^{\perp}},$ so $X$ is a non-zero vector which is orthogonal to $\mathfrak{h}$ and to $\mathfrak{l}.$ This is impossible since $\mathfrak{g}=\mathfrak{h}+\mathfrak{l}.$
\end{proof}
\begin{lemma}
Let $\Lambda_{\mathfrak{h}}:\mathfrak{k}\rightarrow \mathfrak{h}$ be the linear operator defined by
$$\Lambda_{\mathfrak{h}} (X):=proj_{\mathfrak{h}} \circ \Lambda(X).$$ 
The map $\Lambda_{\mathfrak{h}}$ induces an invariant connection in the principal bundle $K\rightarrow K/H$. The  associated bundle 
$$G/L\rightarrow P/L\rightarrow K/H$$
is fat (with respect to an invariant connection) if and only if the bundle 
$$H/(H\cap L)\rightarrow K/L\rightarrow K/H$$
is fat.
\label{lemmabb}
\end{lemma}
\begin{proof} Note that in the proof of this Lemma, for brevity,  we use a terminology which is somewhat not precise but clear: we say that $\Lambda$ or $\Lambda_{\mathfrak{h}}$ is fat, if it yields a fat associated bundle in our usual sense.
The condition $\Lambda_{\mathfrak{h}}(X)=X$ for $X\in \mathfrak{h}$ is obviously satisfied. Also
$$\Lambda([X,Y])=\Lambda_{\mathfrak{h}}([X,Y])+\Lambda_{\mathfrak{h}^{\perp}}([X,Y])$$
and
$$[X,\Lambda (Y)]=[X,\Lambda_{\mathfrak{h}}(Y)+\Lambda_{\mathfrak{h}^{\perp}}(Y)]=[X,\Lambda_{\mathfrak{h}}(Y)]+[X,\Lambda_{\mathfrak{h}^{\perp}}(Y)],$$
for $X\in \mathfrak{h}$ and $Y\in \mathfrak{k}.$ But $[X,\Lambda_{\mathfrak{h}}(Y)]\in \mathfrak{h}$ and $[X,\Lambda_{\mathfrak{h}^{\perp}}(Y)]\in \mathfrak{h}^{\perp}$ since  $K/H$ is reductive and, therefore, we obtain
$$\Lambda_{\mathfrak{h}}([X,Y])=[X,\Lambda_{\mathfrak{h}}(Y)],$$
as desired.

It follows from Lemma \ref{lemma1} that if $\Lambda$ is fat then
$$\forall_{0\neq X\in (\mathfrak{h} \cap \mathfrak{l})^{\perp}_{\mathfrak{h}}} \ \forall_{0\neq Y\in \operatorname{Ker}\Lambda} \ \exists_{Z\in \operatorname{Ker}\Lambda} \ B_{\mathfrak{g}}(X,\Lambda([Y,Z]))\neq 0.$$
But
$$B_{\mathfrak{g}}(X,\Lambda([Y,Z]))=B_{\mathfrak{g}}(X,\Lambda_{\mathfrak{h}}([Y,Z])+\Lambda_{\mathfrak{h}^{\perp}}([Y,Z]))=$$
$$B_{\mathfrak{g}}(X,\Lambda_{\mathfrak{h}}([Y,Z]))+B_{\mathfrak{g}}(X,\Lambda_{\mathfrak{h}^{\perp}}([Y,Z]))=B_{\mathfrak{g}}(X,\Lambda_{\mathfrak{h}}([Y,Z]))$$
and so if $\Lambda$ is fat then
$$\forall_{0\neq X\in (\mathfrak{h} \cap \mathfrak{l})^{\perp}_{\mathfrak{h}}} \ \forall_{0\neq Y\in \operatorname{Ker}\Lambda} \ \exists_{Z\in \operatorname{Ker}\Lambda} \ B_{\mathfrak{g}}(X,\Lambda_{\mathfrak{h}}([Y,Z]))\neq 0.$$
On the other hand if $\Lambda_{\mathfrak{h}}$ is fat then again by Lemma \ref{lemma1} it is fat as a map $\mathfrak{k} \rightarrow \mathfrak{g}$.
\end{proof}
Now we are ready to get the main part of the proof of Theorem \ref{thm:main}. 
\begin{lemma}\label{thm:sub-main} Let $G\rightarrow P\rightarrow K/H$ be a  $G$-structure over a homogeneous space $K/H$ of a compact Lie group $K$. Assume that $G$ is compact. An associated bundle 
$$G/L\rightarrow P/L\rightarrow K/H$$
is fat with respect to a connection metric determined by the $K$-invariant canonical connection if and only if the following three conditions are satisfied:
\begin{enumerate}
\item $G=HL$,
\item $K/H$ is a Riemannian symmetric space;
\item the pair ($\mathfrak{h},\mathfrak{h}\cap\mathfrak{l})$ is a B\'erard-Bergery pair.
\end{enumerate}
\end{lemma}
\begin{proof}
By proposition \ref{prop:onishchik}, we have $\mathfrak{g}=\mathfrak{h}+\mathfrak{l}$. It follows from Lemma \ref{lemmabb} that
$$G/L \rightarrow P/L \rightarrow K/H$$
is fat if and only if the homogeneous bundle $H/H\cap L\rightarrow K/L\rightarrow K/H$ is fat, hence, $K/H$ must be symmetric by Proposition \ref{prop:bb}, the triple $(K,H,L)$ must be the B\'erard-Bergery triple, and the pair $(H,H\cap L)$ must be the B\'erard-Bergery pair.
\end{proof}
\subsection{Proof of Theorem \ref{thm:main}}
We complete the proof in two steps. Assume, first, that $G$ is simple.
If $G$ is simple then $G/L \rightarrow P/L \rightarrow K/H$ is fat if and only if it is B\'erard-Bergery bundle. Indeed, by Lemma \ref{thm:sub-main} $\mathfrak{g}=\mathfrak{h}+\mathfrak{l}$ and the pair $(\mathfrak{h},\mathfrak{h}\cap\mathfrak{l})$ is a B\'erard-Bergery pair. Thus, either $(\mathfrak{g},\mathfrak{h},\mathfrak{l})$ is a factorization or $\mathfrak{g}=\mathfrak{h}$. In the first case, the triple $(\mathfrak{g},\mathfrak{h},\mathfrak{l})$ and $\mathfrak{h}\cap\mathfrak{l}$ are contained in Table 1.
 However, the comparison of Table 1 and Table 2 shows that none of the pairs $(\mathfrak{h},\mathfrak{h}\cap\mathfrak{l})$ is contained in both tables. 
Hence, $G=H$. We get the commutative diagram
$$
\CD
H @>{=}>> H\\
@VVV @VVV\\
K @>{i}>> P\\
@VVV @VVV\\
K/H @>{=}>>K/H
\endCD
$$
where $i$ denotes the evaluation map. Clearly, under the conditions of Theorem \ref{thm:main}, we get an isomorphism of the principal bundles. This follows since the isotropy representation is faithful, and, therefore, $i$ must be an embedding, hence a diffeomorphism. This completes the proof when $G$ is simple. Note that the argument also goes through for any $\mathfrak{g}=\mathfrak{h}$. 
 In the general case, the proof follows from Lemma \ref{lemma:extension} below. We need to consider the case when $\mathfrak{g}$ is not simple, but $(\mathfrak{g},\mathfrak{h},\mathfrak{l})$ is a factorization.  

\begin{lemma}\label{lemma:extension}
Under the assumptions of Theorem \ref{thm:main}, the following holds
$$\mathfrak{g}=\mathfrak{h}'+\mathfrak{l}, \  \mathfrak{g}=\mathfrak{h}'+\mathfrak{n} \ \textrm{and} \ \mathfrak{l}=\mathfrak{h}'+\mathfrak{n}.$$
\end{lemma}
\begin{proof}
We know that $\mathfrak{g}=\mathfrak{h}+\mathfrak{l}$ and $\mathfrak{h}=\mathfrak{h}'+\mathfrak{n}'.$ But $\mathfrak{n}'\subset \mathfrak{h}\cap \mathfrak{l}$ and therefore
$$\mathfrak{g}=\mathfrak{h}'+\mathfrak{l}.$$ 
It follows from Table 2 that $\mathfrak{h}'$ is a simple Lie algebra or $\mathfrak{h}'=\mathbb{R}+\mathfrak{su}(2).$ Take a connected Lie subgroup $H'$ corresponding to $\mathfrak{h}'$ (it exists if $\mathfrak{h}'$ is simple; if $\mathfrak{h}'$ is not simple then $\mathbb{R}$ corresponds to a circle and $H'/(H\cap)L$ is not simply connected). 
So assume that $H'$ is simple. In this case it follows from Table 2 and (see comments to Corollary to Theorem 5.1 and Corollary to Theorem 6.1 in \cite{O}) that only $H'$ acts effectively on $H'/(H\cap L)'$ and therefore $\mathfrak{g}=\mathfrak{h}'+\mathfrak{n}$ and $\mathfrak{l}=\mathfrak{h}'+\mathfrak{n}$ for some ideal $\mathfrak{n} \subset \mathfrak{g}.$
\end{proof}
Now, we complete the proof of Theorem \ref{thm:main} by applying Lemma  \ref{thm:sub-main} and Lemma \ref{lemma:extension}.

\section{Addendum: the B\'erard-Bergery pairs}
\subsection{Table 2}
For the convenience of the reader we reproduce a part of the B\'erad-Bergery classification \cite{BB} (see explanations below Table 2). 
\vskip6pt
\begin{center}
\renewcommand{\arraystretch}{1.2}
\begin{tabular}{|c|c|c|}
\hline
$\mathfrak{h}$ & $\mathfrak{h} \cap \mathfrak{l}$  & restrictions \\
\hline
\hline
$A_{n-2} \oplus \mathbb{R} \oplus A_{1}$ & $A_{n-2} \oplus \mathbb{R} \oplus \mathbb{R}_{1}$ & $n \geq 2$ \\
\hline
$B_{n-2} \oplus A_{1} \oplus A_{1}$ & $B_{n-2} \oplus A_{1} \oplus \mathbb{R}_{1}$ & $n \geq 2$ \\
\hline
$C_{n-1} \oplus A_{1}$ & $C_{n-1} \oplus \mathbb{R}_{1}$ & $n \geq 2$ \\
\hline
$D_{n-2} \oplus A_{1} \oplus A_{1}$ & $D_{n-2} \oplus A_{1} \oplus \mathbb{R}_{1}$ & $n \geq 3$ \\
\hline
$A_{5} \oplus A_{1}$ & $A_{5} \oplus \mathbb{R}_{1}$ & \\
\hline
$D_{6} \oplus A_{1}$ & $D_{6} \oplus \mathbb{R}_{1}$ & \\
\hline
$E_{7} \oplus A_{1}$ & $E_{7} \oplus \mathbb{R}_{1}$ & \\
\hline
$C_{3} \oplus A_{1}$ & $C_{3} \oplus \mathbb{R}_{1}$ & \\
\hline
$\overline{A}_{1} \oplus A_{1}$ & $\overline{A}_{1} \oplus \mathbb{R}_{1}$ & \\
\hline
$A_{1} \oplus \overline{A}_{1}$ & $A_{1} \oplus \mathbb{R}_{1}$ & \\
\hline
$A_{n-4} \oplus \mathbb{R} \oplus A_{3}$ & $A_{n-4} \oplus \mathbb{R} \oplus \overline{C}_{2}$ & $n \geq 4$ \\
\hline
$B_{n-4} \oplus D_{4}$ & $B_{n-4} \oplus \overline{B}_{3}$ & $n \geq 4$ \\
\hline
$C_{n-2} \oplus C_{2}$ & $C_{n-2} \oplus A_{1} \oplus A_{1}$ & $n \geq 3$ \\
\hline
$D_{n-4} \oplus D_{4}$ & $D_{n-4} \oplus \overline{B}_{3}$ & $n\geq 5$ \\
\hline
$\mathbb{R} \oplus D_{5}$ & $\mathbb{R} \oplus B_{4}$ & \\
\hline
$A_{1} \oplus D_{6}$ & $A_{1} \oplus B_{5}$ & \\
\hline
$D_{8}$ & $B_{7}$ & \\
\hline
$B_{4}$ & $D_{4}$ & \\
\hline
$A_{n-2} \oplus \mathbb{R} \oplus A_{1}$ & $A_{n-2} \oplus \mathbb{R}_{a}$ & $n \geq 2$ \\
\hline
$\mathbb{R} \oplus A_{1} \oplus A_{1}$ & $A_{1} \oplus \mathbb{R}_{a}$ & \\ 
\hline
\end{tabular}
\renewcommand{\arraystretch}{1.0}
\end{center}
\vskip6pt
\centerline{TABLE 2}

\subsection{Explanations}
We write down only the B\'erard-Bergery pairs 
$(\mathfrak{h},\mathfrak{h}\cap\mathfrak{l})$ which come from the {\it semisimple} Lie algebras $\mathfrak{k}$. The symbols $A_1$ and $\bar A_1$ denote subalgebras associated to the highest and the lowest root.
Note that our notation differs from that of \cite{BB}: the triples $(\mathfrak{k},\mathfrak{h},\mathfrak{h}\cap \mathfrak{l})$ are denoted by $(\mathfrak{g},\mathfrak{h},\mathfrak{k})$ in \cite{BB}. Also, we have reproduced only cases (A), (C) and (D) from the classification. Cases (B), (E) and (F) are not essential in our considerations. This follows from the description of the classification in \cite{BB}:
\begin{enumerate}
\item case (B) is obtained from case (A) by removing $\mathbb{R}_1$ from $\mathfrak{h}\cap\mathfrak{l}$
\item case (E) consists of the pairs 
$$(A_1'\oplus\mathfrak{h}'\oplus A_1''\oplus\mathfrak{h}'',\bar A_1\oplus\mathfrak{h}'\oplus\mathfrak{h}''),$$
where $(A_1'\oplus\mathfrak{h}',\mathfrak{h}')$ is of type (B), and $(A_1''\oplus\mathfrak{h}'',\mathfrak{h}'')$ is either of type (B), or $(\mathfrak{h},\mathfrak{h}\cap\mathfrak{l})$ has the form $(A_1\oplus A_1,\bar A_1)$, where $\bar A_1$ is a diagonal subalgebra in $A_1\oplus A_1$.
\item case (F) is eliminated by the assumption that $\mathfrak{k}$ is semisimple,
\item the dual triples in case (C) yield the same pair $(\mathfrak{h},\mathfrak{h}\cap\mathfrak{l})$ (see Lemma 12 in \cite{BB}).
\end{enumerate}
\vskip6pt

\begin{remark} {\rm We consider only the case of the canonical connection in a $G$-structure. We conjecture that the same holds for any invariant connection, but the problem seems to be much more difficult. It would be interesting to try some other better known invariant connections, for example, metric connections described in spirit of \cite{AFF}.}
\end{remark}
\vskip6pt
\noindent {\bf Acknowledgement.} We thank Ilka Agricola for answering our questions and Jarek K\c edra for valuable discussions.

MB, AS, AT, AW:  Department of Mathematics and Computer Science
 \vskip6pt
 University of Warmia and Mazury
 \vskip6pt
 S\l\/oneczna 54, 10-710 Olsztyn, Poland
 \vskip6pt
 e-mail addresses:
 \vskip6pt
 MB: mabo@matman.uwm.edu.pl
 \vskip6pt
 AS: anna.szczepkowska@matman.uwm.edu.pl
 \vskip6pt
 AT: tralle@matman.uwm.edu.pl
 \vskip6pt
 AW: awoike@matman.uwm.edu.pl
 \vskip6pt

\end{document}